\documentclass[11pt]{article}
\usepackage{amsmath,amsfonts,amsthm,amssymb,mathtools,enumerate}
\usepackage{mathrsfs,graphicx,color,latexsym, tikz, calc}
\usepackage{hyperref}
\usetikzlibrary{shadows}
\usepackage{graphicx,bm} 
\usepackage{float}

\newtheorem{theorem}{Theorem}
\newtheorem{lemma}{Lemma}

\newtheorem{conjecture}{Conjecture}

\title{\bf \Large{The Second Largest Eigenvalue of Stiffness Matrices of Normalized Complete Frameworks}}

\author{
{\small Tingting Wang,\ \ Lu Lu\footnote{Corresponding author.
\newline{\it \hspace*{5mm}Email addresses:} wttcsu@163.com (T. Wang),  lulugdmath@163.com (L. Lu).}}\\[2mm]
\footnotesize  School of Mathematics and Statistics, Central South University\\ 
\footnotesize Changsha, Hunan, 410083, China}
\date{}
\begin{document}

\maketitle
\begin{abstract}
Let $R(G,p)$ be the normalized rigidity matrix of a framework $(G,p)$
in $\mathbb R^d$, and let
\[
L(G,p)=R(G,p)R(G,p)^{T}
\]
be the associated stiffness matrix. We study the extremal eigenvalues
of $L(K_n,p)$ for complete frameworks whose vertices lie on the unit
sphere and have centroid at the origin.

Our main result shows that, whenever $d\ge2$ and the image of $p$
contains at least three distinct points, the second largest eigenvalue
of $L(K_n,p)$ is exactly $n/2$. This settles the eigenvalue part of a
conjecture of Lew et al. [Israel J. Math. 256, 2023]. We further construct
an infinite family of examples, given by regular polygons embedded in a
two-dimensional subspace, for which the eigenvalue $n/2$ has multiplicity
$2n-4$. Consequently, the multiplicity predicted in the conjecture is not
correct in general. Our results reveal a dichotomy: the value of the
second largest eigenvalue is universal, while its multiplicity is
sensitive to the geometry of the underlying point configuration.\\[1mm]


\end{abstract}
\section{Introduction}

Rigidity theory is concerned with the extent to which a geometric framework is determined by the lengths of its edges. Since the pioneering work of Asimow and Roth~\cite{AR1,AR2}, rigidity theory has developed into an active research area lying at the intersection of discrete geometry, graph theory, algebraic geometry and optimization. Besides its intrinsic mathematical interest, rigidity theory has found numerous applications in structural engineering, robotics, sensor network localization, computer-aided design, molecular biology and materials science. We refer the reader to the monograph of Graver et al.~\cite{GSS}, the surveys of Connelly~\cite{Connelly} and Whiteley~\cite{Whiteley}, and the references therein for comprehensive introductions to the subject.

A fundamental problem in rigidity theory is to determine whether a framework is rigid, either locally or globally, under prescribed edge-length constraints. From an algebraic viewpoint, the local aspect of this question is naturally translated into the study of the rigidity matrix: for a given framework, the kernel of this matrix consists of its infinitesimal motions, and for generic frameworks its rank determines infinitesimal, equivalently local, rigidity. Consequently, the rigidity matrix has become one of the principal algebraic objects in modern rigidity theory, playing an important role in both combinatorial and geometric investigations.

More recently, increasing attention has been devoted to the spectral properties of matrices arising from rigidity theory. Among them, the stiffness matrix, obtained as the Gramian of the normalized rigidity matrix, provides a natural analogue of the graph Laplacian in the geometric setting. As a positive semidefinite matrix, its eigenvalues reflect subtle interactions between the combinatorial structure of the underlying graph and the geometry of the point configuration. In particular, the smallest nonzero eigenvalues give rise to higher-dimensional analogues of algebraic connectivity and provide quantitative measures of rigidity. Understanding these eigenvalues is therefore of independent interest, and has led to recent developments concerning rigidity thresholds for random subgraphs, higher-dimensional algebraic connectivity, and stress-matrix methods in global rigidity; see, for example,~\cite{Connelly,JT,LNPE} and the references therein.

To investigate spectral properties of frameworks, Jord\'an and Tanigawa~\cite{JT} 
introduced a normalized version of the rigidity matrix. Let \((G,p)\) be a 
\(d\)-dimensional framework, where \(G=(V,E)\) is a simple graph and
\[
p:V\rightarrow\mathbb R^d
\]
assigns a point in \(\mathbb R^d\) to each vertex. For every edge \(uv\in E\), define
\[
d_{uv}=
\begin{cases}
\dfrac{p(u)-p(v)}{\|p(u)-p(v)\|}, & p(u)\neq p(v),\\[6pt]
0,& p(u)=p(v),
\end{cases}
\]
and let
\[
\mathbf v_{uv}
=(1_u-1_v)\otimes d_{uv},
\]
where \(\| \cdot\|\) denotes the Euclidean norm in \(\mathbb R^{d}\), 
\(\{1_u\}_{u\in V}\) denotes the standard basis of \(\mathbb R^{|V|}\), 
and \(\otimes\) denotes the Kronecker product. The \emph{normalized rigidity matrix} 
of \((G,p)\) is defined by
\[
R(G,p)=\left(\mathbf v_{uv}\right)_{uv\in E},
\]
whose columns are indexed by the edges of \(G\). The associated \emph{stiffness matrix} is
\[
L(G,p)=R(G,p)R(G,p)^T,
\]
which is a symmetric positive semidefinite matrix of order \(d|V|\).

Stiffness matrices also arise naturally beyond rigidity theory, for instance in the stability 
analysis of truss structures in structural engineering~\cite{CG,CP}. Within rigidity theory, 
they have appeared in certain sufficient conditions for the rigidity and stability of 
frameworks~\cite{CW,CTG}. Unlike the classical rigidity matrix, whose rank is the primary 
object of study in infinitesimal rigidity, the stiffness matrix naturally gives rise to a rich 
spectral theory. Since
\(L(G,p)=R(G,p)R(G,p)^T,
\)
its nonzero eigenvalues coincide with the squares of the singular values of the normalized 
rigidity matrix. Consequently, the spectrum of \(L(G,p)\) reflects not only the infinitesimal 
rigidity of the framework but also the geometry of the underlying point configuration. 
This spectral viewpoint has recently attracted considerable attention and naturally leads to 
extremal eigenvalue problems for stiffness matrices, including questions on eigenvalue bounds, 
multiplicities, and geometric characterizations of frameworks~\cite{JT,LNPE}.

In this paper, we specialize to frameworks of the form \((K_n,p)\), where \(K_n\) denotes 
the complete graph on \(n\) vertices. Throughout the paper, we identify the vertex set of the complete graph $K_n$ with the set $[n]=\{1,2,\ldots,n\}$. Owing to the high degree of symmetry of \(K_n\), the associated stiffness matrix 
exhibits particularly structured spectral properties.
 In particular, 
Jord\'an and Tanigawa proved that its largest eigenvalue has a prescribed value independent 
of both the ambient dimension and the point configuration~\cite{JT}. 

\begin{lemma}[{\cite[Lemma~4.3]{JT}}]\label{lem1-1}
Let $p:[n]\rightarrow\mathbb R^d$ be non-constant.
Then the largest eigenvalue of $L(K_n,p)$ is $n$, and
\[
(p(1),\ldots,p(n))^T
\]
is an eigenvector corresponding to the eigenvalue $n$.
\end{lemma}
 Lemma~\ref{lem1-1} completely determines the largest eigenvalue of the stiffness matrix of a complete framework. It is therefore natural to ask whether other extremal eigenvalues also admit explicit descriptions under suitable geometric assumptions. In particular, the second largest eigenvalue is of special interest from both algebraic and geometric perspectives. From the viewpoint of matrix analysis, it measures the spectral gap of the stiffness matrix and provides finer spectral information beyond the Perron eigenvalue. From the viewpoint of rigidity theory, its value and multiplicity reflect additional geometric constraints imposed by the underlying point configuration. Consequently, understanding the second largest eigenvalue has become a natural continuation of the work of Jord\'an and Tanigawa.

Motivated by this question, Lew et al.~\cite{LNPE} studied frameworks $(K_n,p)$ whose vertices lie on the unit sphere and satisfy $\sum_{i=1}^n p(i)=0$. Under these hypotheses, they proved the following lower bound on the multiplicity of the eigenvalue $n/2$.

\begin{lemma}[{\cite[Proposition~4.2]{LNPE}}]\label{lem1-2}
Let $p:[n]\rightarrow\mathbb R^d$ satisfy
\[
\|p(i)\|=1,\qquad
\sum_{i=1}^{n}p(i)=0,
\]
and assume that the image of $p$ contains at least three distinct points.
Then $n/2$ is an eigenvalue of $L(K_n,p)$ with multiplicity at least $n-1$.
\end{lemma}

This result shows that the eigenvalue $n/2$ plays a distinguished role in the spectrum of the stiffness matrix under the above normalization assumptions. It is therefore natural to ask whether this lower bound on the multiplicity is sharp, and whether $n/2$ also admits an extremal spectral interpretation. Motivated by computational evidence and further theoretical considerations, Lew et al.~\cite{LNPE} proposed the following conjecture.

\begin{conjecture}[\cite{LNPE}]\label{conj1-1}
    Let \( d \geq 3 \), and let \( p: [n] \to \mathbb{R}^d \) such that \( \|p(i)\| = 1 \) for all \( i \in [n] \) and \( \sum_{i=1}^n p(i) = 0 \). Assume that the image of \( p \) is of size at least \(3\). Then, the second largest eigenvalue of \( L(K_n, p) \) is \( n/2 \), and its multiplicity is exactly \( n-1 \).
\end{conjecture}

Conjecture~\ref{conj1-1} naturally gives rise to two closely related questions: one concerning the value of the second largest eigenvalue, and the other concerning its multiplicity. Although these two assertions appear to be closely intertwined, our results show that they have rather different natures. More precisely, we confirm the assertion on the second largest eigenvalue, while demonstrating that the asserted multiplicity does not hold in general. Our first main result determines the second largest eigenvalue exactly, thereby settling the first part of Conjecture~\ref{conj1-1}.

\begin{theorem}\label{thm1-1}
    For \(d\ge 2\) and let \( p: [n] \to \mathbb{R}^d \) such that \( \|p(i)\| = 1 \) for all \( i \in [n] \) and \( \sum_{i=1}^n p(i) = 0 \). Assume that the image of \( p \) is of size at least \(3\). Then the second largest eigenvalue of \( L(K_n, p) \) is \( n/2 \).
\end{theorem}

The proof relies on a careful analysis of the quadratic form associated with the stiffness matrix together with the variational characterization of eigenvalues. The argument is independent of the ambient dimension and applies to arbitrary point configurations satisfying the normalization conditions.

Our second result demonstrates that the multiplicity statement in Conjecture~\ref{conj1-1} does not hold in general. More precisely, we construct an infinite family of normalized frameworks whose second largest eigenvalue is still equal to $n/2$, but whose multiplicity is strictly larger than the conjectured value.

\begin{theorem}\label{thm1-2}
Let $d\ge2$, $n\ge3$, and let $H$ be a two-dimensional subspace of $\mathbb R^d$ with orthonormal basis $\{u,v\}$. Define
\[
p(k)=
\cos\frac{2(k-1)\pi}{n}\,u+
\sin\frac{2(k-1)\pi}{n}\,v,
\text{ for }
1\le k\le n.
\]
Then the spectrum of $L(K_n,p)$ is given by
\[
\operatorname{Spec}(L(K_n,p))
=
\left\{
n^{(1)},
\left(\frac n2\right)^{(2n-4)},
0^{(n(d-2)+3)}
\right\},
\]
where superscripts denote algebraic multiplicities.
\end{theorem}

Theorem~\ref{thm1-2} shows that the multiplicity of the eigenvalue $n/2$ depends delicately on the geometry of the underlying point configuration and cannot, in general, be determined solely by the normalization conditions. Consequently, Conjecture~\ref{conj1-1} is only partially true: while the predicted value of the second largest eigenvalue is correct, the proposed multiplicity formula fails even for highly symmetric point configurations. In this sense, our results reveal an unexpected dichotomy between the location of the second largest eigenvalue and its multiplicity.

The remainder of the paper is organized as follows. In Section~2, we establish several auxiliary results concerning the stiffness matrix and develop the spectral tools needed in the subsequent proofs. Section~3 is devoted to the proof of Theorem~\ref{thm1-1}. In Section~4, we determine the complete spectrum of the stiffness matrix associated with regular polygons and prove Theorem~\ref{thm1-2}. We conclude the paper with several remarks and possible directions for further research.
\section{Preliminaries}

In this section we collect several elementary facts that will be used
throughout the paper. We write $p_i:=p(i)$ for $i\in[n]$,
and, when no confusion is possible, we write $R=R(K_n,p)$ and $L=L(K_n,p)$. For a vector
\[
\mathbf{x}=(x_1,\ldots,x_n)^T\in(\mathbb R^d)^n,
\]
we regard each component \(x_i\) as a vector in \(\mathbb R^d\).

\begin{lemma}\label{lem:basic-formulas}
Let \(p:[n]\to\mathbb R^d\), and let \(L=L(K_n,p)\). Then, for every
\(\mathbf{x}=(x_1,\ldots,x_n)^T\in(\mathbb R^d)^n\), we have
\[
(R^T\mathbf{x})_{ij}
=
\langle x_i-x_j,d_{ij}\rangle,
\qquad 1\le i<j\le n,
\]
and hence
\[
\mathbf{x}^TL\mathbf{x}
=
\sum_{1\le i<j\le n}
\langle x_i-x_j,d_{ij}\rangle^2.
\]
Moreover, the \(i\)-th component of \(L\mathbf{x}\) is given by
\[
(L\mathbf{x})_i
=
\sum_{j\ne i}
d_{ij}d_{ij}^T(x_i-x_j).
\]
\end{lemma}

\begin{proof}
The first identity follows directly from the definition of the
normalized rigidity matrix, since the column corresponding to the edge
\(ij\) is \((1_i-1_j)\otimes d_{ij}\). Therefore
\[
(R^T\mathbf{x})_{ij}
=
\langle (1_i-1_j)\otimes d_{ij},\mathbf{x}\rangle
=
\langle x_i-x_j,d_{ij}\rangle.
\]
Since \(L=RR^T\), we obtain
\[
\mathbf{x}^TL\mathbf{x}
=
\mathbf{x}^TRR^T\mathbf{x}
=
\|R^T\mathbf{x}\|^2
=
\sum_{1\le i<j\le n}
\langle x_i-x_j,d_{ij}\rangle^2.
\]
The componentwise formula follows by expanding
\[
L=\sum_{1\le i<j\le n}
\bigl((1_i-1_j)\otimes d_{ij}\bigr)
\bigl((1_i-1_j)\otimes d_{ij}\bigr)^T.
\]
This gives
\[
(L\mathbf{x})_i
=
\sum_{j\ne i}
d_{ij}d_{ij}^T(x_i-x_j),
\]
as claimed.
\end{proof}

\begin{lemma}\label{lem:centered-identities}
Assume that $\|p_i\|=1$ for all $i\in[n]$ and $\sum_{i=1}^n p_i=0$.
Let $c_{ij}:=\langle p_i,p_j\rangle$. Then for every \(i\in[n]\), we have 
\[
\sum_{j=1}^n c_{ij}=0,
\text{ and }
\sum_{j\ne i}(1-c_{ij})=n.
\]
Moreover, for all \(i,j\in[n]\), we have
\[
\|p_i-p_j\|^2=2(1-c_{ij}).
\]
\end{lemma}

\begin{proof}
For each fixed \(i\), we have
\[
\sum_{j=1}^n c_{ij}
=
\sum_{j=1}^n\langle p_i,p_j\rangle
=
\left\langle p_i,\sum_{j=1}^n p_j\right\rangle
=
0.
\]
Since \(c_{ii}=\|p_i\|^2=1\), it follows that
\[
\sum_{j\ne i}(1-c_{ij})
=
(n-1)-\sum_{j\ne i}c_{ij}
=
(n-1)-(-1)
=
n.
\]
Finally, we get
\[
\|p_i-p_j\|^2
=
\|p_i\|^2+\|p_j\|^2-2\langle p_i,p_j\rangle
=
2(1-c_{ij}).
\]
\end{proof}

\begin{lemma}\label{lem:pairwise-distance}
Let \(u_1,\ldots,u_n\) be vectors in a real inner product space. Then
\[
\sum_{1\le i<j\le n}\|u_i-u_j\|^2
=
n\sum_{i=1}^n\|u_i\|^2
-
\left\|\sum_{i=1}^n u_i\right\|^2.
\]
Equivalently,
\[
\frac12\sum_{1\le i<j\le n}\|u_i-u_j\|^2
=
\frac n2\sum_{i=1}^n\|u_i\|^2
-
\frac12
\left\|\sum_{i=1}^n u_i\right\|^2.
\]
\end{lemma}

\begin{proof}
Expanding the left-hand side gives
\[
\sum_{1\le i<j\le n}\|u_i-u_j\|^2
=
\frac12\sum_{i,j=1}^n
\left(
\|u_i\|^2+\|u_j\|^2-2\langle u_i,u_j\rangle
\right).
\]
Hence
\[
\sum_{1\le i<j\le n}\|u_i-u_j\|^2
=
n\sum_{i=1}^n\|u_i\|^2
-
\sum_{i,j=1}^n\langle u_i,u_j\rangle
=
n\sum_{i=1}^n\|u_i\|^2
-
\left\|\sum_{i=1}^n u_i\right\|^2.
\]
This proves the identity.
\end{proof}

\begin{lemma}\label{lem:positive-M}
Assume that $\|p_i\|=1$ for all $i\in[n]$, and that the image of \(p\) contains at least three distinct points.
Then the matrix
\[
M:=nI-\sum_{i=1}^n p_ip_i^T
\]
is positive definite.
\end{lemma}

\begin{proof}
For any \(w\in\mathbb R^d\), we have
\[
w^TMw
=
\sum_{i=1}^n
\left(
\|w\|^2-\langle w,p_i\rangle^2
\right)
\ge 0
\]
by the Cauchy--Schwarz inequality. Hence \(M\) is positive semidefinite.

Suppose that \(w^TMw=0\) for some nonzero \(w\). Then equality must hold
in each term, so every \(p_i\) is parallel to \(w\). Since \(\|p_i\|=1\),
each \(p_i\) must be equal to either \(w/\|w\|\) or \(-w/\|w\|\). Thus
the image of \(p\) contains at most two distinct points, contradicting
the assumption. Therefore \(M\) is positive definite.
\end{proof}

\begin{lemma}\label{lem:spectral-criterion}
Let \(A\) be a real symmetric matrix, and suppose that \(v\neq 0\) is
an eigenvector of \(A\) corresponding to the eigenvalue \(\lambda_1\).
Assume that
\[
x^TAx\le \alpha\|x\|^2
\]
for every \(x\perp v\), where \(\alpha<\lambda_1\). Then every eigenvalue
of \(A\) other than \(\lambda_1\) is at most \(\alpha\). In particular,
\(\lambda_1\) is simple. In addition,  if \(\alpha\) is an eigenvalue of
\(A\), then the second largest eigenvalue of \(A\) is \(\alpha\).
\end{lemma}

\begin{proof}
Since \(A\) is symmetric, there exists an orthonormal basis consisting
of eigenvectors of \(A\). Every eigenvector corresponding to an
eigenvalue different from \(\lambda_1\) is orthogonal to \(v\). Hence, if
\(x\perp v\) is such an eigenvector with eigenvalue \(\lambda\), then
\[
\lambda\|x\|^2=x^TAx\le \alpha\|x\|^2,
\]
so \(\lambda\le\alpha\). Since \(\alpha<\lambda_1\), no other
linearly independent eigenvector can correspond to \(\lambda_1\), and
therefore \(\lambda_1\) is simple. If \(\alpha\) itself is an eigenvalue,
then it is necessarily the second largest eigenvalue.
\end{proof}

\section{A quadratic-form estimate and proof of Theorem~\ref{thm1-1}}

Throughout this section, we assume that
$\|p_i\|=1$, $i\in[n]$ for all $\sum_{i=1}^n p_i=0$, 
and that the image of \(p\) contains at least three distinct points.
We write $p_i:=p(i)$, $R:=R(K_n,p)$ and $L:=L(K_n,p)$.

We first prove an auxiliary estimate for vectors that orthogonal with $p_i$.

\begin{lemma}\label{lem:tangent-estimate}
Let \(y_1,\ldots,y_n\in\mathbb R^d\) satisfy
$\langle y_i,p_i\rangle=0$ for all $i\in[n]$.
Then we have
\[
\sum_{1\le i<j\le n}
\langle y_i-y_j,d_{ij}\rangle^2
\le
\frac12
\sum_{1\le i<j\le n}
\|y_i-y_j\|^2 .
\]
\end{lemma}

\begin{proof}
For \(\mathbf y=(y_1,\ldots,y_n)^T\in(\mathbb R^d)^n\), define
\[
D(\mathbf y):=
\frac12\sum_{1\le i<j\le n}\|y_i-y_j\|^2
-
\sum_{1\le i<j\le n}\langle y_i-y_j,d_{ij}\rangle^2 .
\]It is equivalent to prove that \(D(\mathbf y)\ge0\).

We first reduce to the case $\sum_{i=1}^n y_i=0$. Let
$Y:=\sum_{i=1}^n y_i$ and $M:=nI-\sum_{i=1}^n p_ip_i^T$. By Lemma~\ref{lem:positive-M}, the matrix \(M\) is positive definite.
Set
\[
u:=M^{-1}Y,
~
t_i:=u-\langle u,p_i\rangle p_i,
\text{ and }
y_i':=y_i-t_i .
\]
Then $\langle t_i,p_i\rangle=0$, $\langle y_i',p_i\rangle=0$,
and
\[
\sum_{i=1}^n y_i'
=
Y-\sum_{i=1}^n
\bigl(u-\langle u,p_i\rangle p_i\bigr)
=
Y-Mu
=
0.
\]

We claim that $D(\mathbf y')=D(\mathbf y)$. For simplicity, we write $a_i:=\langle u,p_i\rangle$.
Then $t_i=u-a_ip_i$ and $t_i-t_j=-a_ip_i+a_jp_j$.
By direct calculations, we get
\[
D(\mathbf y-\mathbf t)=D(\mathbf y)+D(\mathbf t)-E(\mathbf y,\mathbf t),
\]
where $E(\mathbf y,\mathbf t)
=
\sum_{i<j}\langle y_i-y_j,t_i-t_j\rangle
-
2\sum_{i<j}
\langle y_i-y_j,d_{ij}\rangle
\langle t_i-t_j,d_{ij}\rangle$.

We first show that \(E(\mathbf y,\mathbf t)=0\). For each pair \(i<j\), a direct
calculation gives
\[
\langle y_i-y_j,t_i-t_j\rangle
-
2\langle y_i-y_j,d_{ij}\rangle
 \langle t_i-t_j,d_{ij}\rangle=
-a_i\langle y_i,p_j\rangle
-a_j\langle y_j,p_i\rangle .
\]
Indeed, if \(p_i=p_j\), then \(a_i=a_j\), \(t_i=t_j\), and both sides
are equal to zero. If \(p_i\ne p_j\), the identity follows from
$d_{ij}=\frac{p_i-p_j}{\|p_i-p_j\|}$, $\|p_i-p_j\|^2=2(1-\langle p_i,p_j\rangle)$, 
together with $\langle y_i,p_i\rangle=\langle y_j,p_j\rangle=0$.
Summing over all pairs, we obtain
\[
E(\mathbf y,\mathbf t)
=
-\sum_{i=1}^n
a_i
\sum_{j\ne i}\langle y_i,p_j\rangle .
\]
Since \(\sum_{j=1}^n p_j=0\), we have $\sum_{j\ne i}\langle y_i,p_j\rangle
=
\left\langle y_i,-p_i\right\rangle
=
0$. Thus \(E(\mathbf y,\mathbf t)=0\).

Next we prove \(D(\mathbf t)=0\). For every pair \(i<j\), using
\(c_{ij}:=\langle p_i,p_j\rangle\), we have
\[
\frac12\|t_i-t_j\|^2
-
\langle t_i-t_j,d_{ij}\rangle^2
=
\frac12\left(c_{ij}(a_i^2+a_j^2)-2a_ia_j\right).
\]
This identity is also valid when \(p_i=p_j\), since then \(a_i=a_j\)
and both sides are zero. Therefore, we have
\[
D(\mathbf t)=\frac12
\sum_{i<j}
\left(c_{ij}(a_i^2+a_j^2)-2a_ia_j\right)=
\frac12
\sum_{i=1}^n a_i^2\sum_{j\ne i}c_{ij}
-
\sum_{i<j}a_ia_j .
\]
By Lemma~\ref{lem:centered-identities}, we have $\sum_{j\ne i}c_{ij}=-1$. Hence,
\[
D(\mathbf t)
=
-\frac12\sum_{i=1}^n a_i^2-\sum_{i<j}a_ia_j
=
-\frac12\left(\sum_{i=1}^n a_i\right)^2.
\]
Since $\sum_{i=1}^n a_i
=
\sum_{i=1}^n\langle u,p_i\rangle
=
\left\langle u,\sum_{i=1}^n p_i\right\rangle
=
0$, we get \(D(\mathbf t)=0\), and consequently
\(
D(\mathbf y')=D(\mathbf y).
\)

It remains to prove \(D(\mathbf y')\ge0\). Since \(\sum_i y_i'=0\), it suffices
to work with the centered family \(y_1',\ldots,y_n'\). For each \(i\),
define$
B_i:=
\frac1{\sqrt2}
\bigl(p_i\otimes y_i'-y_i'\otimes p_i\bigr),
$
where \(\mathbb R^d\otimes\mathbb R^d\) is endowed with the Frobenius
inner product. Since \(\|p_i\|=1\) and
\(\langle y_i',p_i\rangle=0\), we have $\|B_i\|=\|y_i'\|$.

We claim that
\(
\langle y_i'-y_j',d_{ij}\rangle^2
\le
\frac{1+c_{ij}}2\|B_i-B_j\|^2\) for all \(i<j\).
If \(p_i=p_j\), then \(d_{ij}=0\), and the inequality is trivial.
Assume \(p_i\ne p_j\). Let
$
A_{ij}:=
\frac1{\sqrt2}
\bigl(p_i\otimes p_j-p_j\otimes p_i\bigr).
$
Then, we get $\|A_{ij}\|^2=1-c_{ij}^2$ and a direct computation gives
\[
\langle B_i-B_j,A_{ij}\rangle
=
\langle y_i',p_j\rangle+\langle y_j',p_i\rangle .
\]
Moreover, we have $\langle y_i'-y_j',d_{ij}\rangle^2
=
\frac{\left(\langle y_i',p_j\rangle+\langle y_j',p_i\rangle\right)^2
}{2(1-c_{ij})}$.
By the Cauchy--Schwarz inequality, we get
\[
\langle y_i'-y_j',d_{ij}\rangle^2
=\frac{\langle B_i-B_j,A_{ij}\rangle^2}{2(1-c_{ij})}\le
\frac{\|B_i-B_j\|^2\|A_{ij}\|^2}{2(1-c_{ij})}=
\frac{1+c_{ij}}2\|B_i-B_j\|^2.
\]

Summing this inequality over all pairs yields
\[
\sum_{i<j}\langle y_i'-y_j',d_{ij}\rangle^2
\le
\frac12\sum_{i<j}(1+c_{ij})\|B_i-B_j\|^2.
\]
We now use the centeredness of \(p_1,\ldots,p_n\). Since
$
\sum_{j=1}^n c_{ij}=0
$
for every \(i\),  and $
\sum_{i=1}^n c_{ij}=0
$ for every \(j\), we have
\[
\begin{aligned}
\sum_{i<j}c_{ij}\|B_i-B_j\|^2
&=
\frac12\sum_{i,j}c_{ij}\|B_i-B_j\|^2  \\
&=
\frac12\sum_{i,j}c_{ij}(\|B_i\|^2+\|B_j\|^2-2\langle B_i,B_j\rangle)\\
& =\frac12\left(\sum_{i}\|B_i\|^2\sum_{j}c_{ij}+\sum_{j}\|B_j\|^2 \sum_{i}c_{ij}\right)-\sum_{i,j}c_{ij}\langle B_i, B_j\rangle\\
&=-\sum_{i,j}c_{ij}\langle B_i, B_j\rangle\\
&=-\left\|\sum_{i=1}^n p_i\otimes B_i\right\|^2
\le0.
\end{aligned}
\]
Therefore, we get
\[
\sum_{i<j}\langle y_i'-y_j',d_{ij}\rangle^2
\le
\frac12\sum_{i<j}\|B_i-B_j\|^2.
\]
By Lemma~\ref{lem:pairwise-distance}, we have
\[
\frac12\sum_{i<j}\|B_i-B_j\|^2
=
\frac n2\sum_{i=1}^n\|B_i\|^2
-
\frac12\left\|\sum_{i=1}^n B_i\right\|^2
\le
\frac n2\sum_{i=1}^n\|B_i\|^2.
\]
Using \(\|B_i\|=\|y_i'\|\), we obtain
$
\sum_{i<j}\langle y_i'-y_j',d_{ij}\rangle^2
\le
\frac n2\sum_{i=1}^n\|y_i'\|^2.
$
Since \(\sum_i y_i'=0\), Lemma~\ref{lem:pairwise-distance} gives
$
\frac12\sum_{i<j}\|y_i'-y_j'\|^2
=
\frac n2\sum_{i=1}^n\|y_i'\|^2.
$
Thus, we obtain
\[
\sum_{i<j}\langle y_i'-y_j',d_{ij}\rangle^2
\le
\frac12\sum_{i<j}\|y_i'-y_j'\|^2,
\]
which means \(D(\mathbf y')\ge0\). Since \(D(\mathbf y')=D(\mathbf y)\), we also have
\(D(\mathbf y)\ge0\). This proves the lemma.
\end{proof} 

Next, we give the key lemma on estimating the quadratic form $\mathbf{x}^TL\mathbf{x}$.

\begin{lemma}\label{lem:key-estimate}
For every $\mathbf{x}=(x_1,\ldots,x_n)^T\in(\mathbb R^d)^n$,
we have
\[
\mathbf{x}^TL\mathbf{x}
\le
\frac n2\sum_{i=1}^n\|x_i\|^2
+
\frac12
\left(
\sum_{i=1}^n\langle x_i,p_i\rangle
\right)^2
-\frac12
\left\|
\sum_{i=1}^n x_i
\right\|^2 .
\]
\end{lemma}

\begin{proof}
For each \(i\in[n]\), set
$
a_i:=\langle x_i,p_i\rangle\ \text{and}\ 
y_i:=x_i-a_ip_i.$
Then $
x_i=a_ip_i+y_i,\ 
\langle y_i,p_i\rangle=0,\ \text{and}\ 
\|x_i\|^2=a_i^2+\|y_i\|^2.
$
For \(1\le i<j\le n\),  put $
c_{ij}:=\langle p_i,p_j\rangle\ \text{and}\ 
z_{ij}:=\langle y_i,p_j\rangle+\langle y_j,p_i\rangle .$
Since $\|p_i-p_j\|^2=2(1-c_{ij})$,
we have, for \(p_i\ne p_j\),
\[
\begin{aligned}
\langle x_i-x_j,d_{ij}\rangle^2
&=
\frac{
\left(
(1-c_{ij})(a_i+a_j)-z_{ij}
\right)^2
}
{2(1-c_{ij})}  \\
&=
\frac12(1-c_{ij})(a_i+a_j)^2
-(a_i+a_j)z_{ij}
+
\langle y_i-y_j,d_{ij}\rangle^2 .
\end{aligned}
\]
If \(p_i=p_j\), then \(d_{ij}=0\), \(c_{ij}=1\) and
\(z_{ij}=0\), so the same identity remains valid. Hence, by
Lemma~\ref{lem:basic-formulas}, we get 
\[
\mathbf{x}^TL\mathbf{x}=T_1-T_2+T_3,
\]
where
\[
T_1:=
\frac12\sum_{i<j}(1-c_{ij})(a_i+a_j)^2,
\]
\[
T_2:=
\sum_{i<j}(a_i+a_j)z_{ij},
\]
and
\[
T_3:=
\sum_{i<j}\langle y_i-y_j,d_{ij}\rangle^2 .
\]

We estimate these three terms separately. First, expanding \(T_1\) and
using Lemma~\ref{lem:centered-identities}, we obtain
\[
\begin{aligned}
T_1
&=
\frac12\sum_{i=1}^n a_i^2
\sum_{j\ne i}(1-c_{ij})
+
\sum_{i<j}(1-c_{ij})a_ia_j  \\
&=
\frac n2\sum_{i=1}^n a_i^2
+
\sum_{i<j}(1-c_{ij})a_ia_j .
\end{aligned}
\]
Moreover,
\[
\sum_{i<j}(1-c_{ij})a_ia_j
=
\frac12\left(\sum_{i=1}^n a_i\right)^2
-
\frac12\left\|\sum_{i=1}^n a_ip_i\right\|^2.
\]
Therefore,
\[
T_1
=
\frac n2\sum_{i=1}^n a_i^2
+
\frac12\left(\sum_{i=1}^n a_i\right)^2
-
\frac12\left\|\sum_{i=1}^n a_ip_i\right\|^2.
\]

Second, using $
\sum_{j\ne i}\langle y_i,p_j\rangle
=
\left\langle y_i,-p_i\right\rangle
=
0$, 
we get
\[
T_2 =
\sum_{i=1}^n a_i\sum_{j\ne i}z_{ij}=\sum_{i=1}^{n}a_i\sum_{j\ne i}\left(\langle y_i,p_j\rangle+\langle y_j,p_i\rangle     \right)=\sum_{i=1}^{n}a_i\sum_{j=1}^n \langle y_j,p_i\rangle =
\left\langle
\sum_{j=1}^n y_j,
\sum_{i=1}^n a_ip_i
\right\rangle .
\]

Third, by Lemma~\ref{lem:tangent-estimate} and
Lemma~\ref{lem:pairwise-distance},
\[
T_3\le
\frac12\sum_{i<j}\|y_i-y_j\|^2 =
\frac n2\sum_{i=1}^n\|y_i\|^2
-
\frac12\left\|\sum_{i=1}^n y_i\right\|^2.
\]

Combining the estimates for \(T_1,T_2,T_3\), we obtain
\[
\begin{aligned}
\mathbf{x}^TL\mathbf{x}
&=
T_1-T_2+T_3  \\
&\le
\frac n2\sum_{i=1}^n a_i^2
+
\frac12\left(\sum_{i=1}^n a_i\right)^2
-
\frac12\left\|\sum_{i=1}^n a_ip_i\right\|^2  \\
&\quad
-
\left\langle
\sum_{j=1}^n y_j,
\sum_{i=1}^n a_ip_i
\right\rangle
+
\frac n2\sum_{i=1}^n\|y_i\|^2
-
\frac12\left\|\sum_{i=1}^n y_i\right\|^2  \\
&=
\frac n2\sum_{i=1}^n(a_i^2+\|y_i\|^2)
+
\frac12\left(\sum_{i=1}^n a_i\right)^2 
-
\frac12
\left\|
\sum_{i=1}^n a_ip_i+\sum_{i=1}^n y_i
\right\|^2  
.\end{aligned}
\]
Since $
a_i^2+\|y_i\|^2=\|x_i\|^2$, $ x_i=a_ip_i+y_i$ and $
a_i=\langle x_i,p_i\rangle$, 
the desired inequality follows.
\end{proof}

Finally, we present the proof of Theorem~\ref{thm1-1}.
\begin{proof}[Proof of Theorem~\ref{thm1-1}]
Let
\[
\mathbf p:=(p_1,\ldots,p_n)^T\in(\mathbb R^d)^n.
\]
Since the image of \(p\) contains at least three distinct points, \(p\)
is non-constant. Hence Lemma~\ref{lem1-1} implies that the largest
eigenvalue of \(L=L(K_n,p)\) is \(n\), with eigenvector \(\mathbf p\).

Let \(\mathbf x=(x_1,\ldots,x_n)^T\in(\mathbb R^d)^n\) be orthogonal to
\(\mathbf p\). Then $
\sum_{i=1}^n\langle x_i,p_i\rangle
=
\langle \mathbf x,\mathbf p\rangle
=
0.$
By Lemma~\ref{lem:key-estimate}, we get
\[
\mathbf x^TL\mathbf x
\le
\frac n2\sum_{i=1}^n\|x_i\|^2
=
\frac n2\|\mathbf x\|^2 .
\]
Thus, by Lemma~\ref{lem:spectral-criterion}, every eigenvalue of \(L\)
other than \(n\) is at most \(n/2\).

Finally, Lemma~\ref{lem1-2} shows that \(n/2\) is an eigenvalue of
\(L(K_n,p)\). Since \(n\) is the largest eigenvalue and \(n>n/2\), the
second largest eigenvalue of \(L(K_n,p)\) is exactly \(n/2\).
\end{proof}

\section{Regular polygons and proof of Theorem~\ref{thm1-2}}

Throughout this section, we assume that \(d\ge2\), \(n\ge3\), and that
\(H\) is a two-dimensional subspace of \(\mathbb R^d\) with orthonormal
basis \(\{u,v\}\). We write
\[
p_k:=p(k)
=
\cos\frac{2(k-1)\pi}{n}\,u
+
\sin\frac{2(k-1)\pi}{n}\,v, \text{ for } 1\le k\le n.
\]
Let $
L:=L(K_n,p).$ Since \(p_i-p_j\in H\) for all \(i,j\), every direction vector \(d_{ij}\)
also lies in \(H\).

\subsection{Reduction to the planar component}

We first separate the contribution from \(H\) and its orthogonal
complement. Since $
\mathbb R^d=H\oplus H^\perp$, we have the orthogonal direct sum decomposition $
(\mathbb R^d)^n=H^n\oplus (H^\perp)^n$.

\begin{lemma}\label{lem:regular-splitting}
The subspace \(H^n\) is invariant under \(L\), and \(L\) vanishes on
\((H^\perp)^n\). In particular, $
(H^\perp)^n\subseteq \ker L$.
\end{lemma}

\begin{proof}
Since \(p_i\in H\) for all \(i\), we have $
p_i-p_j\in H $ for all \(i,j\). Hence each direction vector \(d_{ij}\) belongs to \(H\)
as well, including the case \(d_{ij}=0\).
Using the componentwise formula for the stiffness matrix, for \(\mathbf x\in (\mathbb R^d)^n\), we have
\[
(L\mathbf x)_i
=
\sum_{j\ne i}
d_{ij}d_{ij}^T(x_i-x_j),
\qquad i=1,\ldots,n .
\]
We first prove that \(H^n\) is invariant under \(L\). Let $\mathbf h=(h_1,\ldots,h_n)^T\in H^n$. 
Then \(h_i-h_j\in H\) for all \(i,j\). Since \(d_{ij}\in H\), we have
\[
d_{ij}d_{ij}^T(h_i-h_j)
=
\langle h_i-h_j,d_{ij}\rangle d_{ij}
\in H .
\]
Therefore each component \((L\mathbf h)_i\) belongs to \(H\), and hence $L\mathbf h\in H^n$.
Thus \(H^n\) is invariant under \(L\).
We next show that \(L\) vanishes on \((H^\perp)^n\). Let $
\mathbf x=(x_1,\ldots,x_n)^T\in (H^\perp)^n$. 
Then \(x_i-x_j\in H^\perp\) for all \(i,j\), while \(d_{ij}\in H\). Hence
\[
d_{ij}^T(x_i-x_j)
=
\langle x_i-x_j,d_{ij}\rangle
=
0 .
\]
Consequently, $
d_{ij}d_{ij}^T(x_i-x_j)=0 $
for every \(i\ne j\). It follows from the componentwise formula that $
L\mathbf x=0$.
Therefore \(L\) vanishes on \((H^\perp)^n\), and in particular $
(H^\perp)^n\subseteq \ker L$.
\end{proof}
Consequently, the spectrum of \(L\) is obtained from the spectrum of
the restriction $
L_H:=L|_{H^n}:H^n\to H^n$ 
together with the additional zero eigenvalue contribution coming from
\((H^\perp)^n\), whose dimension is \(n(d-2)\). It remains to determine
the spectrum of \(L_H\).

\subsection{Complex representation of the planar component}

We identify the real plane \(H\) with the complex plane by the isometric
isomorphism
\[
\Phi:H\to\mathbb C,\ \text{and}\ 
\Phi(au+bv)=a+b\mathrm{i},
\qquad a,b\in\mathbb R.
\]
Under this identification, $
q_k:=\Phi(p_k)
=
e^{2(k-1)\pi\mathrm{i}/n},\  k=1,\ldots,n$. Moreover, $
\sum_{k=1}^n q_k=0$.
For \(i\ne j\), let
\[
P_{ij}:H\to H,\ 
\text{and}\ 
P_{ij}(x)=d_{ij}d_{ij}^T x
=
\langle x,d_{ij}\rangle d_{ij},
\]
be the orthogonal projection onto the line spanned by \(d_{ij}\).
Let \(P_{ij}'\) denote the corresponding real-linear operator on
\(\mathbb C\), namely $
P_{ij}'=\Phi\circ P_{ij}\circ \Phi^{-1}$.

\begin{lemma}\label{lem:projection-complex}
 For every \(z\in\mathbb C\) and every \(i\ne j\), we have \[ P_{ij}'(z) = \frac12\left(z-q_iq_j\overline z\right). \] 
 \end{lemma} 
\begin{proof} 
Let \[ w:=\Phi(d_{ij}) = \frac{q_i-q_j}{|q_i-q_j|}. \] Then \(|w|=1\), and 
\[ P_{ij}'(z) = \Phi\circ P_{ij}\circ \Phi^{-1}(z)=\langle \Phi^{-1}(z),d_{ij}\rangle\Phi (d_{ij})=\langle z,w\rangle_{\mathbb R}w, \]
 where $ \langle z,w\rangle_{\mathbb R} =\operatorname{Re}(z\overline w)$ is the standard real inner product on \(\mathbb C\). Hence \[ P_{ij}'(z) = \operatorname{Re}(z\overline w)w = \frac12(z\overline w+\overline z\,w)w = \frac12(z+\overline z\,w^2). \] It remains to compute \(w^2\). 
 Since $ q_i=e^{\mathrm{i}\theta_i},\ q_j=e^{\mathrm{i}\theta_j}$, 
 where $ \theta_k=\frac{2(k-1)\pi}{n}$, 
 we have \[ q_i-q_j = 2\mathrm{i}\, e^{\mathrm{i}(\theta_i+\theta_j)/2} \sin\frac{\theta_i-\theta_j}{2}. \] 
 Therefore \[ w^2 = \frac{(q_i-q_j)^2}{|q_i-q_j|^2} = -q_iq_j. \] 
 Substituting this into the expression for \(P_{ij}'(z)\), we obtain \[ P_{ij}'(z) = \frac12\left(z-q_iq_j\overline z\right). \]
\end{proof}

Let \(\Phi^n: H^n \to \mathbb C^n\) be the componentwise extension of \(\Phi\), defined by $
\Phi^n\left((x_1,\dots,x_n)\right)=\left(\Phi(x_1),\dots,\Phi(x_n)\right)$.
We define the real-linear
operator $ \mathcal L:\mathbb C^n\to\mathbb C^n$ by \[\mathcal L=\Phi^n\circ L_H\circ(\Phi^n)^{-1}.\]
Thus \(\mathcal L\) is the complex-coordinate representation of \(L_H\).
For \(x=(x_1,\ldots,x_n)^T\in\mathbb C^n\), set $
S(x):=\sum_{k=1}^n x_k$ and $
T(x):=\sum_{k=1}^n q_k\overline{x_k}$.

\begin{lemma}\label{lem:regular-L-formula}
For every \(x=(x_1,\ldots,x_n)^T\in\mathbb C^n\), we have
\[
(\mathcal Lx)_i
=
\frac n2 x_i
-
\frac12 S(x)
+
\frac12 q_iT(x),
\qquad i=1,\ldots,n.
\]
\end{lemma}

\begin{proof}
By the definition of \(L\), let \(\mathbf{y}\in (\mathbb{R}^{d})^n\). Then for each \(i\),
$(L\mathbf y)_i
=
\sum_{j\ne i}
d_{ij}d_{ij}^T(y_i-y_j).
$
After applying \(\Phi\) componentwise, we obtain
\[
(\mathcal Lx)_i
=
\sum_{j\ne i}P_{ij}'(x_i-x_j).
\]
Using Lemma~\ref{lem:projection-complex}, this becomes
\[
(\mathcal Lx)_i
=
\frac12\sum_{j\ne i}(x_i-x_j)
-
\frac12\sum_{j\ne i}q_iq_j\overline{x_i-x_j}.
\]

The first sum is
\[
\frac12\sum_{j\ne i}(x_i-x_j)
=
\frac n2x_i-\frac12S(x).
\]
For the second sum, we compute
\[
\begin{aligned}
\sum_{j\ne i}q_iq_j\overline{x_i-x_j}
&=
q_i
\left(
\overline{x_i}\sum_{j\ne i}q_j
-
\sum_{j\ne i}q_j\overline{x_j}
\right).
\end{aligned}
\]
Since $
\sum_{j=1}^n q_j=0$, 
we have $
\sum_{j\ne i}q_j=-q_i$ and $
\sum_{j\ne i}q_j\overline{x_j}
= T(x)-q_i\overline{x_i}.$
Therefore
\[
\sum_{j\ne i}q_iq_j\overline{x_i-x_j}=
q_i
\left(
-q_i\overline{x_i}
-
T(x)
+
q_i\overline{x_i}
\right) =
-q_iT(x).
\]
Substituting this into the preceding expression gives
\[
(\mathcal Lx)_i
=
\frac n2x_i-\frac12S(x)+\frac12q_iT(x),
\]
as claimed.
\end{proof}

\subsection{Eigenspace decomposition}

We now regard \(\mathbb C^n\) as a real vector space of dimension \(2n\). Let $
\mathbf 1:=(1,\ldots,1)^T$ and $
\mathbf q:=(q_1,\ldots,q_n)^T$.
Define the following real subspaces of \(\mathbb C^n\): \[
W_1:=\{c\mathbf 1\mid c\in\mathbb{C}\}, \qquad W_2:=\{r\mathbf q \mid r\in\mathbb R\}, \qquad 
W_3:=\{r\mathrm{i}\mathbf q\mid r\in\mathbb R\},\]
and
\[
W_4:=\left\{x=(x_1,\ldots,x_n)^\top\in\mathbb{C}^n \;\middle|\; \sum_{i=1}^{n}x_i=0,\; \sum_{i=1}^{n}q(i)\bar{x}_i=0\right\}.
\]

\begin{lemma}\label{lem:regular-decomposition}
As a real vector space, we have
\[
\mathbb C^n
=
W_1\oplus W_2\oplus W_3\oplus W_4.
\]
Moreover, their dimensions are given by
\[
\dim_{\mathbb R}W_1=2,\qquad
\dim_{\mathbb R}W_2=1,\qquad
\dim_{\mathbb R}W_3=1,\ \text{and} \ \dim_{\mathbb R}W_4=2n-4.
\]
\end{lemma}

\begin{proof}
First we show that \(W_1,W_2,W_3\) are linearly independent over
\(\mathbb R\). Suppose that
\[
c\mathbf 1+r\mathbf q+s\mathrm{i}\mathbf q=0,
\]
where \(c\in\mathbb C\) and \(r,s\in\mathbb R\). Then for every
\(k=1,\ldots,n\),
\[
c+(r+s\mathrm{i})q_k=0.
\]
Since the numbers \(q_1,\ldots,q_n\) are distinct, this implies $
r+s\mathrm{i}=0 $
and \(c=0\). Hence $
W_1\oplus W_2\oplus W_3$ is a direct sum, with real dimension \(2+1+1=4\).

We use the standard real inner product on \(\mathbb C^n\),
$
\langle x,y\rangle_{\mathbb R}
=
\operatorname{Re}\sum_{k=1}^n x_k\overline{y_k}.
$ For \(x\in\mathbb C^n\), the condition \(x\perp W_1\) is equivalent to $
\sum_{k=1}^n x_k=0,
$ that is, \(S(x)=0\). Similarly, since $
W_2\oplus W_3=\{\alpha\mathbf q:\alpha\in\mathbb C\}$, 
the condition \(x\perp W_2\oplus W_3\) is equivalent to $
\sum_{k=1}^n x_k\overline{q_k}=0$, i.e.  equivalent to
$\sum_{k=1}^n q_k\overline{x_k}=0$, that is, \(T(x)=0\). Therefore $
W_4=(W_1\oplus W_2\oplus W_3)^\perp $.
It follows that
\[
\mathbb C^n
=
W_1\oplus W_2\oplus W_3\oplus W_4
\]
and
\[
\dim_{\mathbb R}W_4=2n-4.
\]
\end{proof}

\begin{lemma}\label{lem:regular-eigen-decomposition}
The operator \(\mathcal L\) acts on the four subspaces above as follows:
\[
\mathcal L|_{W_1}=0,\qquad
\mathcal L|_{W_2}=n\,\mathrm{Id},\qquad
\mathcal L|_{W_3}=0,\ \text{and}\ 
\mathcal L|_{W_4}=\frac n2\,\mathrm{Id}.
\]
Consequently,
\[
\operatorname{Spec}(\mathcal L)
=
\left\{
n^{(1)},
\left(\frac n2\right)^{(2n-4)},
0^{(3)}
\right\}.
\]
\end{lemma}

\begin{proof}
We apply Lemma~\ref{lem:regular-L-formula}.

First let \(x\in W_1\). Then $
x=c\mathbf 1$
for some \(c\in\mathbb C\). Hence, we have $
S(x)=nc$ 
and $T(x)=\sum_{k=1}^n q_k\overline c
=
\overline c\sum_{k=1}^n q_k
=
0.$
Therefore, we have
\[
(\mathcal Lx)_i
=
\frac n2c-\frac12nc+0
=
0,
\]
so \(\mathcal Lx=0\). Hence, we have $
\mathcal L|_{W_1}=0$. 

Next let \(x\in W_2\). Then
$ x=r\mathbf q$ 
for some \(r\in\mathbb R\). We have $S(x)=r\sum_{k=1}^n q_k=0$
and $
T(x)
=
\sum_{k=1}^n q_k\overline{rq_k}
=
r\sum_{k=1}^n |q_k|^2
=
rn.
$
Thus, we get
\[
(\mathcal Lx)_i
=
\frac n2rq_i+\frac12q_i(rn)
=
nrq_i
=
nx_i.
\]
Hence, we get $
\mathcal L|_{W_2}=n\,\mathrm{Id}$.

Now let \(x\in W_3\). Then $
x=r\mathrm{i}\mathbf q
$ for some \(r\in\mathbb R\). Hence $S(x)=r\mathrm{i}\sum_{k=1}^n q_k=0$
and
$T(x)
=
\sum_{k=1}^n q_k\overline{r\mathrm{i}q_k}
=
-r\mathrm{i}\sum_{k=1}^n |q_k|^2
=
-r\mathrm{i}n.
$
Therefore, we get
\[
(\mathcal Lx)_i
=
\frac n2r\mathrm{i}q_i
+
\frac12q_i(-r\mathrm{i}n)
=
0.
\]
Thus, we have
$
\mathcal L|_{W_3}=0.
$

Finally, if \(x\in W_4\), then by definition $
S(x)=0$ and $T(x)=0$.
Lemma~\ref{lem:regular-L-formula} gives $
(\mathcal Lx)_i
=
\frac n2x_i$. 
Thus, we get $
\mathcal L|_{W_4}=\frac n2\,\mathrm{Id}$.

Together with Lemma~\ref{lem:regular-decomposition}, this gives a full
eigenspace decomposition of \(\mathcal L\) as a real-linear operator on
\(\mathbb C^n\). Since
\[
\dim_{\mathbb R}W_1=2,\qquad
\dim_{\mathbb R}W_2=1,\qquad
\dim_{\mathbb R}W_3=1,\ \text{and}\ 
\dim_{\mathbb R}W_4=2n-4,
\]
we obtain
\[
\operatorname{Spec}(\mathcal L)
=
\left\{
n^{(1)},
\left(\frac n2\right)^{(2n-4)},
0^{(3)}
\right\}.
\]
\end{proof}

We can now prove Theorem~\ref{thm1-2}.

\begin{proof}[Proof of Theorem~\ref{thm1-2}]
By Lemma~\ref{lem:regular-splitting}, we have the invariant orthogonal
decomposition $(\mathbb R^d)^n
=
H^n\oplus (H^\perp)^n$, 
where \(L\) vanishes on \((H^\perp)^n\). Since $
\dim (H^\perp)^n=n(d-2)$, 
the subspace \((H^\perp)^n\) contributes \(n(d-2)\) additional zero
eigenvalues.

On the planar part \(H^n\), the operator \(L_H=L|_{H^n}\) is conjugate,
via the real-linear isomorphism \(\Phi^n:H^n\to\mathbb C^n\), to the
operator \(\mathcal L\). Therefore \(L_H\) and \(\mathcal L\) have the
same spectrum. By Lemma~\ref{lem:regular-eigen-decomposition},
\[
\operatorname{Spec}(L_H)
=
\left\{
n^{(1)},
\left(\frac n2\right)^{(2n-4)},
0^{(3)}
\right\}.
\]
Combining this with the zero eigenspace \((H^\perp)^n\), we obtain
\[
\operatorname{Spec}(L(K_n,p))
=
\left\{
n^{(1)},
\left(\frac n2\right)^{(2n-4)},
0^{(n(d-2)+3)}
\right\}.
\]
This proves the theorem.
\end{proof}

\section{Concluding remarks}

In this paper we studied the spectrum of the stiffness matrix of
normalized complete frameworks. Our main result shows that, under the
assumptions
\[
\|p(i)\|=1,\qquad \sum_{i=1}^n p(i)=0,
\]
and provided that the image of \(p\) contains at least three distinct
points, the second largest eigenvalue of \(L(K_n,p)\) is always equal to
\(n/2\). Thus the eigenvalue part of Conjecture~\ref{conj1-1} holds, in
fact, for all \(d\ge2\).

On the other hand, the multiplicity of this eigenvalue behaves quite
differently. By analyzing regular \(n\)-gons embedded in a
two-dimensional subspace, we obtained the full spectrum
\[
\operatorname{Spec}(L(K_n,p))
=
\left\{
n^{(1)},
\left(\frac n2\right)^{(2n-4)},
0^{(n(d-2)+3)}
\right\}.
\]
This shows that the multiplicity of the eigenvalue \(n/2\) can be
strictly larger than \(n-1\) when \(n\ge 3\), and hence the multiplicity assertion in
Conjecture~\ref{conj1-1} is not valid in general.

The results reveal a natural dichotomy in the spectral theory of
stiffness matrices of complete frameworks: the location of the second
largest eigenvalue is universal under the normalization conditions,
whereas its multiplicity is sensitive to the geometry of the underlying
point configuration. It would be interesting to characterize those
normalized point configurations for which the multiplicity of \(n/2\) is
exactly \(n-1\), and more generally to determine all possible
multiplicities of this eigenvalue. Another natural direction is to
investigate whether analogous phenomena occur for other families of
graphs or for different normalization conditions on the framework.
{

}
\end{document}